\documentclass{amsart}
\usepackage{amssymb,latexsym}
\usepackage{amscd,amsthm}

\newtheorem{theorem}{Theorem}[section]
\newtheorem{lemma}[theorem]{Lemma}
\newtheorem{proposition}[theorem]{Proposition}
\newtheorem{corollary}[theorem]{Corollary}

\theoremstyle{definition}
\newtheorem{definition}[theorem]{Definition}

\newtheorem*{remark}{Remark}

\newtheorem*{question}{Question}

\DeclareMathOperator{\Ext}{Ext}
\DeclareMathOperator{\Hom}{Hom}

\DeclareMathOperator{\Tor}{Tor}

\DeclareMathOperator{\cok}{cok}

\newcommand{\tensor}{\otimes}

\newcommand{\class}[1]{\mathcal{#1}}   

\newcommand{\Z}{\mathbb{Z}}
\newcommand{\Q}{\mathbb{Q/Z}}
\newcommand{\mathcolon}{\colon\,} 
\newcommand{\ar}{\xrightarrow{}} 

\newcommand{\rightperp}[1]{#1^{\perp}}
\newcommand{\leftperp}[1]{{}^\perp #1}

\begin{document}

\title{Model structures on modules over Ding-Chen rings}

\date{\today}

\author{James Gillespie}
\thanks{The author thanks Nanqing Ding of Nanjing University for his interest and support in writing this paper and for sharing his most recent work.}
\address{505 Ramapo Valley Road \\
         Ramapo College of New Jersey \\
         Mahwah, NJ 07430}
\email[Jim Gillespie]{jgillesp@ramapo.edu}


\begin{abstract}
An $n$-FC ring is a left and right coherent ring whose left and right self FP-injective dimension is $n$. The work of Ding and Chen in~\cite{ding and chen 93} and~\cite{ding and chen 96} shows that these rings possess properties which generalize those of $n$-Gorenstein rings. In this paper we call a (left and right) coherent ring with finite (left and right) self FP-injective dimension a Ding-Chen ring. In case the ring is Noetherian these are exactly the Gorenstein rings. We look at classes of modules we call Ding projective, Ding injective and Ding flat which are meant as analogs to Enochs' Gorenstein projective, Gorenstein injective and Gorenstein flat modules. We develop basic properties of these modules. We then show that each of the standard model structures on Mod-$R$, when $R$ is a Gorenstein ring, generalizes to the Ding-Chen case. We show that when $R$ is a commutative Ding-Chen ring and $G$ is a finite group, the group ring $R[G]$ is a Ding-Chen ring.
\end{abstract}

\maketitle


\section{Introduction}

Non-commutative Gorenstein rings were defined and studied by Y.~Iwanaga in~\cite{iwanaga} and~\cite{iwanaga2}. Later Enochs and Jenda and coauthors defined and studied the so-called Gorenstein injective, Gorenstein projective and Gorenstein flat modules and developed Gorenstein homological algebra. In~\cite{hovey} we saw that this theory can be formalized in the language of model categories. There, Hovey showed that when $R$ is a Gorenstein ring, the category of $R$-modules has two Quillen equivalent model structures, a projective model structure and an injective model structure. Their homotopy categories are what we call the stable module category of the ring $R$. There is also a third Quillen equivalent model structure called the flat model structure, as was shown in~\cite{gillespie-hovey}.

The point of this paper is to describe how the homotopy theory on modules over a Gorenstein ring generalizes to a homotopy theory on modules over a so-called $n$-FC ring, or what we call, a Ding-Chen ring. While an $n$-Gorenstein ring is a Noetherian ring with self injective dimension $n$, an $n$-FC ring is a coherent ring with self FP-injective dimension $n$. These rings were introduced and studied by Ding and Chen in~\cite{ding and chen 93} and~\cite{ding and chen 96} and seen to have many properties similar to $n$-Gorenstein rings. Just as a ring is called Gorenstein when it is $n$-Gorenstein for some natural number $n$, we will call a ring ``Ding-Chen'' when it is $n$-FC for some $n$. (The term FC ring is already taken! In the language of Ding and Chen, these are the 0-FC rings.) By now, the work of Ding and coauthors is sophisticated enough to readily show that each of Hovey's model structures generalizes to modules over a Ding-Chen ring. In fact, Ding and Mao explicitly remark in~\cite{ding and mao 07} that one of these model structures (the one we call the injective model structure) must exist.

We can describe the basic idea rather simply in terms of cotorsion pairs, which are intimately related to model category structures by~\cite{hovey}. When $R$ is a Gorenstein ring, we have the class of ``trivial'' modules $\class{W}$, which are the modules of finite injective dimension. It turns out that a module is trivial iff it has finite projective dimension iff it has finite flat dimension. Hovey's injective model structure relies on the fact that $(\class{W},\class{G}\class{I})$ is a complete cotorsion pair, where $\class{G}\class{I}$ is Enochs' class of Gorenstein injective modules. In this model structure every module is cofibrant and $\class{G}\class{I}$ makes up the class of fibrant modules. On the other hand, Hovey's projective model structure relies on the fact that $(\class{G}\class{P},\class{W})$ is a complete cotorsion pair, where $\class{G}\class{P}$ is Enochs' class of Gorenstein projective modules. In this dual model structure every module is fibrant, while $\class{G}\class{P}$ makes up the class of cofibrant modules. When $R$ is a Ding-Chen ring, the class $\class{W}$ of trivial modules is relaxed to consist of the modules of finite FP-injective dimension. It now turns out that a module is trivial iff it has finite flat dimension. Now the injective model structure relies on the fact that $(\class{W},\class{D}\class{I})$ is a complete cotorsion pair, where $\class{D}\class{I}$ is the class of Ding injective modules (fibrant modules). On the other hand, the projective model structure relies on the fact that $(\class{D}\class{P},\class{W})$ is a complete cotorsion pair, where $\class{D}\class{P}$ is the class of Ding projective modules (cofibrant modules). Furthermore, when $R$ is Noetherian, a Ding-Chen ring is automatically Gorenstein. In this case $\class{D}\class{I} = \class{G}\class{I}$ and $\class{D}\class{P} = \class{G}\class{P}$ and our model structures coincide with Hovey's in~\cite{hovey}.

There is a similar result involving a flat model structure on modules over a Ding-Chen ring. In the process we introduce and look at first properties of Ding injective, Ding projective and Ding flat modules. They are the obvious generalizations of Gorenstien injective, Gorenstein projective and Gorenstein flat modules to the Ding-Chen ring case. However, during the writing of this paper the author was made aware that Ding injective modules have appeared in~\cite{ding and mao 08} under the name \emph{Gorenstein FP-injective modules}, while Ding projective modules have appeared in~\cite{ding and mao and li 09} under the name \emph{strongly Gorenstein flat modules}. The author feels that the ``Ding'' names are quite fitting and illuminate the analogy with the homotopy theory of modules over a Gorenstein ring.

The layout of the paper is as follows: In Section~\ref{sec-preliminaries} we give definitions and review concepts which are basic to the rest of the paper. In Section~\ref{sec-ding modules} we introduce Ding modules and use techniques of Enochs and coauthors to show that Ding modules have properties analogous to Gorenstein modules. Similar results have appeared in~\cite{ding and mao 08} and~\cite{ding and mao and li 09}. In Section~\ref{sec-injective and projective model strucs} we define Ding-Chen rings and show the existence of the injective, projective, and flat model structures. This follows rather easily from the work of Ding and coauthors. In Section~\ref{sec-applications to group rings} we show that when $R$ is a commutative Ding-Chen ring and $G$ is a finite group, the group ring $R[G]$ is a Ding-Chen ring. The model structures on $R[G]$ are relevant for defining Tate cohomology.

\section{Preliminaries}\label{sec-preliminaries}

Throughout we assume all rings have an identity and all modules are unital. Unless stated otherwise, an $R$-module will be understood to be a \emph{right} $R$-module. We denote the injective (resp. projective, resp. flat) dimension of an $R$-module $M$ by id($M$) (resp. pd($M$) resp. fd($M$)).

If $R$ is commutative and Noetherian and of finite injective dimension when viewed as a module over itself, then $R$ is called a Gorenstein ring. When $R$ is non commutative, Iwanaga extended the definition of Gorenstein rings as follows: $R$ is (left and right) Noetherian and both id($_RR$) and id($R_R$) are finite. In this case Iwanaga showed that id($_RR$) = id($R_R$), and if this number is $n$, we say $R$ is $n$-Gorenstein. The book~\cite{enochs-jenda-book} is a standard reference for Gorenstein rings and modules.

In~\cite{ding and chen 93} and especially~\cite{ding and chen 96}, Ding and Chen extended this idea yet further by replacing (left and right) Noetherian with (left and right) coherence and replacing the (left and right) injective dimension of $R$ with the (left and right) FP-injective dimension of $R$. The FP-injective dimension of a module was introduced by Stenstr\"{o}m in~\cite{stenstrom-FP-injective} and will be defined in Section~\ref{sec-ding modules}. The rings introduced by Ding and Chen (called $n$-FC rings, or Ding-Chen rings) will be defined in Section~\ref{sec-injective and projective model strucs}.
Examples of Ding-Chen rings include all Gorenstein rings, and the group rings of Section~\ref{sec-applications to group rings}. Any von-Neumann regular ring is also a Ding-Chen ring. In particular, if $R$ is an infinite product of fields, then $R$ is a Ding-Chen ring. Furthermore it follows from Theorem~7.3.1 of~\cite{glaz} that $R[x_1,x_2,x_3, \cdots , x_n]$ is a commutative Ding-Chen ring. Another example of a Ding-Chen ring is the group ring $R[G]$ where $R$ is an FC-ring and $G$ is a locally finite group. See~\cite{damiano}.

A \emph{cotorsion pair} (of $R$-modules) is a pair of classes of modules
$(\class{A},\class{B})$ such that
$\rightperp{\class{A}} = \class{B}$ and $\class{A} =
\leftperp{\class{B}}$. Here $\rightperp{\class{A}}$ is the class
of modules $M$ such that $\Ext^1_R(A,M) = 0$ for all $A
\in \class{A}$, and similarly $\leftperp{\class{B}}$ is the class
of modules $M$ such that $\Ext^1_R(M,B) = 0$ for all $B
\in \class{B}$. Two simple examples of cotorsion
pairs are $(\class{P},\class{A})$ and
$(\class{A},\class{I})$ where $\class{P}$ is the class of
projectives, $\class{I}$ is the class of injectives and
$\class{A}$ is the class of all $R$-modules. A cotorsion pair $(\class{A},\class{B})$ is said to have \emph{enough projectives} if for
any module $M$ there is a short exact sequence $0
\xrightarrow{} B \xrightarrow{} A \xrightarrow{} M \xrightarrow{}
0$ where $B \in \class{B}$ and $A \in \class{A}$. We say it has
\emph{enough injectives} if it satisfies the dual statement. These
two statements are in fact equivalent for the category of $R$-modules.
We say that the cotorsion pair is
\emph{complete} if it has enough projectives and injectives. Note that the cotorsion pairs $(\class{P},\class{A})$
and $(\class{A},\class{I})$ above are each complete.
The book~\cite{enochs-jenda-book} is also an excellent reference for cotorsion pairs.
The equivalence of the enough injectives and enough projectives, although not difficult,
is proved as Proposition 7.1.7 in ~\cite{enochs-jenda-book}.

The most well-known (nontrivial) example of a cotorsion pair is
$(\class{F},\class{C})$ where $\class{F}$ is the class of flat
modules and $\class{C}$ are the cotorsion modules. A proof that this is a complete cotorsion theory can be found in~\cite{enochs-jenda-book}.

\begin{definition}\label{def-hereditary}
A cotorsion pair $(\class{A},\class{B})$ is
called \emph{hereditary} if one of the following hold:
  \begin{enumerate}
  \item $\class{A}$ is resolving. That is, $\class{A}$ is closed
  under taking kernels of epis.
  \item $\class{B}$ is coresolving. That is, $\class{B}$ is closed
  under taking cokernels of monics.
  \item $\Ext_R^{i}(A,B) = 0$ for any $R$-modules $A \in \class{A}$
  and $B \in \class{B}$ and $i \geq 1$.
  \end{enumerate}
\end{definition}

See \cite{enochs-jenda-book} for a proof that these are equivalent.

Complete hereditary cotorsion pairs are intimately related to abelian model category structures and we refer the reader to Theorem~2.2 of~\cite{hovey} for the precise relationship. We end this section by proving a lemma that will be used frequently in the rest of the paper.

\begin{lemma}\label{lemma-intersecting cotorsion classes}
Let $\class{I}$ be the class of injective modules and $\class{P}$ be the class of projective modules. If $\class{W}$ is a coresolving class of objects containing $\class{I}$ and $(\class{W}, \rightperp{\class{W}})$ is a cotorsion pair, then $\class{W} \cap \rightperp{\class{W}} = \class{I}$. On the other hand, if $\class{W}$ is resolving and contains $\class{P}$ and $(\leftperp{\class{W}}, \class{W})$ is a cotorsion pair then $\leftperp{\class{W}} \cap \class{W} = \class{P}$.
\end{lemma}

\begin{proof}
The two statements are dual. We prove the first one. Clearly $\class{I} \subseteq \class{W} \cap \rightperp{\class{W}}$. On the other hand, let $W \in \class{W} \cap \rightperp{\class{W}}$, and write $0 \ar W \ar I \ar C \ar 0$ where $I$ is injective. Since $\class{W}$ contains the injectives and is coresolving, $C$ must also be in $\class{W}$. Now $\Ext^1_R(C,W) = 0$, so the sequence must split. So $W$ is a summand of $I$ proving $W$ is injective.
\end{proof}

\section{Ding modules}\label{sec-ding modules}

In this section we consider modules we call Ding injective, Ding projective and Ding flat modules. Just as the Gorenstein modules play the role of fibrant and cofibrant objects in the model structures on modules over Gorenstein rings, the Ding modules play the role of fibrant and cofibrant objects over Ding-Chen rings. Ding injective and projective modules just appeared in the literature in~\cite{ding and mao 08} and~\cite{ding and mao and li 09} respectively.

\begin{definition}
A right $R$-module $E$ is a called \emph{FP-injective} if $\Ext^1_R(F,E) = 0$ for all finitely presented modules $F$. More generally, the \emph{FP-injective dimension} of a right $R$-module $N$ is defined to be the least integer $n \geq 0$ such that $\Ext^{n+1}_R(F,N) = 0$ for all finitely presented right $R$-modules $F$. The FP-injective dimension of $N$ is denoted FP-id($N$) and equals $\infty$ if no such $n$ above exists.
\end{definition}

 These definitions were introduced in~\cite{stenstrom-FP-injective}. There it is shown (Lemma~3.1) that for a (right) coherent ring, the notion of FP-id($N$) behaves analogously to id($N$). In particular, FP-id($N$) $\leq n$ iff $\Ext^{n+1}_R(F,N) = 0$ for all finitely presented $F$ iff $\Ext^{n+1}_R(R/I,N) = 0$ for all finitely generated (right) ideals of $R$ iff the $n$th cosyzygy of any FP-injective coresolution of $N$ is FP-injective. Note that the notion of FP-injective coincides with the notion of injective when $R$ is Noetherian.

\subsection{Ding injectives}

We now introduce and look at basic properties of Ding injective modules. They also appear in~\cite{ding and mao 08} as ``Gorenstein FP-injective modules''. We refer the reader to~\cite{ding and mao 08} for more results, including results concerning existence of covers and envelopes.

\begin{definition} We call a right $R$-module $N$ \emph{Ding injective} if there exists an exact sequence of injective modules $$\cdots \ar I_1 \ar I_0 \ar I^0 \ar I^1 \ar \cdots$$ with $N = \ker{(I^0 \ar I^1)}$ and which remains exact after applying $\Hom_R(E,-)$ for any FP-injective module $E$. We denote the class of Ding injective modules by $\class{D}\class{I}$.
\end{definition}

\begin{remark}
By definition, Ding injective modules are Gorenstein injective. When $R$ is Noetherian the two notions coincide.
\end{remark}

\begin{lemma}\label{lemma-first Ext lemma}
If $E$ is FP-injective and $N$ is Ding injective then $\Ext^i_R(E,N) = 0$ for each $i \geq 1$.
\end{lemma}

\begin{proof}
In the definition of Ding injective, $N$ has an injective coresolution $0 \ar N \ar I^0 \ar I^1 \ar \cdots$ which remains exact after applying $\Hom_R(E,-)$ for any FP-injective module $E$. So $\Ext^i_R(E,N) = 0$ for each $i \geq 1$.
\end{proof}

\begin{proposition}\label{prop-FP-injective dimension is 0 or inf}
A Ding injective module is either injective or has FP-injective dimension $\infty$.
\end{proposition}

\begin{proof}
Suppose $N$ is Ding injective and $\Ext^{n+1}_R(F,N) = 0$ for all finitely presented $F$.
Let $0 \ar N \ar I^0 \ar I^1 \ar \cdots \ar I^{n-1} \ar E^n$ be exact with each $I^i$ injective and $E^n = \cok{(I^{n-2} \ar I^{n-1})}$. It follows from the isomorphism $\Ext^1_R(F,E_n) \cong \Ext^{n+1}_R(F,N)$, that $E^n$ is FP-injective. Now $0 \ar N \ar I^0 \ar I^1 \ar \cdots \ar I^{n-1} \ar E^n \ar 0$ represents an element of $\Ext^n_R(E^n,N)$, but this group equals 0 by Lemma~\ref{lemma-first Ext lemma}. Therefore the sequence is split exact and so $N$ is a direct sum of $I^0$, and so must be injective.
\end{proof}

\begin{corollary}
If the class of Ding injective right $R$-modules is closed under direct sums, then $R$ is right Noetherian.
\end{corollary}

\begin{proof}
Suppose the class of Ding injectives is closed under direct sums. Then since the class of FP-injective modules is always closed under direct sums by Corollary~2.4 of~\cite{stenstrom-FP-injective}, it follows that a direct sum of injectives must be both Ding injective and FP-injective. But then it must be injective by Proposition~\ref{prop-FP-injective dimension is 0 or inf}. So $R$ is (right) Noetherian.
\end{proof}

\subsection{Ding projectives}

We now introduce a kind of dual notion, that of Ding projective module. To be more precise, we will see in Section~\ref{sec-injective and projective model strucs} that the duality holds when $R$ is a Ding-Chen ring. Ding projective modules just appeared in~\cite{ding and mao and li 09} as ``strongly Gorenstein flat modules''. We refer the reader to~\cite{ding and mao and li 09} for more results, in particular, for results concerning existence of covers and envelopes.

\begin{definition} We call a right $R$-module $M$ \emph{Ding projective} if there exists an exact sequence of projective modules $$\cdots \ar P_1 \ar P_0 \ar P^0 \ar P^1 \ar \cdots$$ with $M = \ker{(P^0 \ar P^1)}$ and which remains exact after applying $\Hom_R(-,F)$ for any flat module $F$. We denote the class of all Ding projective modules by $\class{D}\class{P}$.
\end{definition}

\begin{remark}
By definition, Ding projective modules are Gorenstein projective. For a general coherent ring $R$, it follows from Proposition~10.2.6 of~\cite{enochs-jenda-book} that a finitely presented module is Ding projective if and only if it is Gorenstein projective. Also, from  Corollary~\ref{cor-characterization of FC projectives} the Ding projectives are exactly the Gorenstein projectives when $R$ is a Gorenstein ring.
\end{remark}

\begin{lemma}\label{lemma-second Ext lemma}
If $F$ is flat and $M$ is Ding projective then $\Ext^i_R(M,F) = 0$ for each $i \geq 1$.
\end{lemma}

\begin{proposition}\label{prop-flat dimension is either 0 of infty}
A Ding projective module is either projective or has flat dimension $\infty$.
\end{proposition}

\begin{proof}
Suppose $M$ is Ding projective and fd($M$) = $n < \infty$. Then we can construct an exact sequence $0 \ar F_n \ar P_{n-1} \ar \cdots \ar P_1 \ar P_0 \ar M \ar 0$ with each $P_i$ projective and $F_n = \ker{(P_{n-1} \ar P_{n-2})}$ flat. Then this exact sequence is an element of $\Ext^n_R(M,F_n)$ which equals 0 by Lemma~\ref{lemma-second Ext lemma}. Therefore the resolution is split exact and so $M$ must be projective.
\end{proof}

\begin{corollary}
\begin{enumerate}
\item If the Ding projective right $R$-modules are closed under direct limits then $R$ is right perfect. \\

\item If $R$ is left coherent and the Ding projective right $R$-modules are closed under direct products then $R$ is right perfect. \\
\end{enumerate}
\end{corollary}

\begin{proof}
Recall that a ring $R$ is right perfect if and only if every flat right $R$-module is projective. So by Lazard's theorem it follows that $R$ is right perfect if and only if the class of projective right $R$-modules is closed under direct limits. Suppose the class of Ding projective right $R$-modules is closed under direct limits. Then it follows that a direct limit of projectives must be both Ding projective and flat. But then it must be projective by Proposition~\ref{prop-flat dimension is either 0 of infty}. So $R$ is right perfect.

For the second statement use the results of S.U.~Chase: First, a ring $R$ is left coherent if and only if direct products of flat right $R$-modules are flat. Second, a ring $R$ is left coherent and right perfect if and only if direct products of projective right $R$-modules are projective. Then the proof is similar to the last paragraph.
\end{proof}

\subsection{Ding flats}

One could introduce the notion of a ``Ding flat'' module but it turns out that they are nothing more than the Gorenstein flat modules.

\begin{definition} Call a left $R$-module $M$ \emph{Ding flat} if there exists an exact sequence of flat modules $$\cdots \ar F_1 \ar F_0 \ar F^0 \ar F^1 \ar \cdots$$ with $M = \ker{(F^0 \ar F^1)}$ and which remains exact after applying $E \tensor_R -$ for any FP-injective right $R$-module $E$.
\end{definition}

The following result is a restatement of Ding and Mao's Lemma~2.8 of~\cite{ding and mao 08}.

\begin{proposition}
A left $R$-module $M$ is Ding flat if and only if it is Gorenstein flat. In this case $M^+$ is a Ding injective right $R$-module. The converse holds when $R$ is a right coherent ring.
\end{proposition}

We refer the reader to Theorem~5 of~\cite{ding and chen 96} for characterizations of Ding flat (i.e., Gorenstein flat) modules over Ding-Chen rings.

Over a Gorenstein ring $R$, it is known that Gorenstein projective modules are Gorenstein flat. (See Corollary~10.3.10
of~\cite{enochs-jenda-book}.) However it is not known whether or not this holds over more general rings $R$. Affirmative answers have been given for some other types of rings. For example, see Proposition~3.4 of~\cite{henrik holm-Goren dimension}. Of course, one could ask the analogous question: Are Ding projective modules Ding flat? While this is true when $R$ is a Ding-Chen ring, the following shows that it holds in much more generality.

\begin{proposition}\label{prop-ding projective implies ding flat}
Let $R$ be a left coherent ring. Then any Ding projective right $R$-module $M$ is Ding flat.
\end{proposition}

\begin{proof}
Let $M$ be a Ding projective right $R$-module. By definition there is an exact sequence of projective right $R$-modules $$\cdots \ar P_1 \ar P_0 \ar P^0 \ar P^1 \ar \cdots$$ with $M = \ker{(P^0 \ar P^1)}$ and which remains exact after applying $\Hom_R(-,F)$ for any flat right $R$-module $F$. Let $E$ be an FP-injective left $R$-module. Since $R$ is left coherent, $E^+$ is a flat right $R$-module by Theorem~2.2 of~\cite{fieldhouse}. So applying $\Hom_R(-,E^+)$ we get the exact sequence $$\cdots \ar \Hom(P^1, E^+) \ar \Hom(P^0, E^+) \ar \Hom(P_0, E^+) \ar \Hom(P_1, E^+) \ar \cdots$$
But this is naturally isomorphic to $$\cdots \ar (P^1 \tensor_R E)^+ \ar (P^0 \tensor_R E)^+ \ar (P_0  \tensor_R E)^+ \ar (P_1 \tensor_R E)^+ \ar \cdots$$ Therefore $\cdots \ar P_1 \tensor_R E \ar P_0 \tensor_R E \ar P^0  \tensor_R E \ar P^1 \tensor_R E \ar \cdots$ is exact, proving $M$ is a Ding flat right $R$-module.
\end{proof}

In considering the above results one is led to the following question which we do not have an answer for.

 \begin{question}
 When $R$ is a Ding-Chen ring, are there Gorenstein injective modules that are not Ding injective?
\end{question}


\section{model structures on modules over a Ding-Chen ring}\label{sec-injective and projective model strucs}

In~\cite{stenstrom-FP-injective}, Stenstr\"{o}m introduced the notion of an FP-injective module and studied FP-injective modules over coherent rings.
FC rings appear to have been introduced by Damiano in~\cite{damiano}. These are the ``coherent version'' of quasi-Frobenius rings where Noetherian is replaced by coherent and self injective is replaced by self FP-injective. Finally, just as Gorenstein rings are a natural generalization of quasi-Frobenius rings, Ding and Chen extended FC rings to $n$-FC rings in~\cite{ding and chen 93} and~\cite{ding and chen 96}. Throughout this section modules are \emph{right} $R$-modules unless stated otherwise.

\begin{definition}
A ring $R$ is called an \emph{$n$-FC ring} if it is both left and right coherent and FP-id($_RR$) = FP-id($R_R$) = $n$. Often we are not interested in the particular $n$, and so we say $R$ is a \emph{Ding-Chen ring} if it is an $n$-FC ring for some $n \geq 0$.
\end{definition}

We note that Corollary~3.18 of~\cite{ding and chen 93} states that if $R$ is both left and right coherent, and FP-id($_RR$) and FP-id($R_R$) are both finite, then FP-id($_RR$) = FP-id($R_R$) . Ding and Chen go on to prove the following fundamental result, which is the key reason we have a homotopy theory on the category $\text{Mod-}R$ when $R$ is a Ding-Chen ring.

 \begin{theorem}\label{theorem-Ding and Chen}
 Let $R$ be an $n$-FC ring and $M$ be a right $R$-module. Then the following are equivalent:
 \begin{enumerate}
 \item $\text{fd}(M)  < \infty$
 \item $\text{fd}(M)  \leq n$
 \item $\text{FP-id}(M) < \infty$
 \item $\text{FP-id}(M) \leq n$
 \end{enumerate}
 \end{theorem}

 Due to Theorem~\ref{theorem-Ding and Chen} we can make the following definition.

 \begin{definition}
 If $R$ is a Ding-Chen ring and $M$ is an $R$-module, we say that $M$ is \emph{trivial} if it satisfies one of the equivalent conditions of Theorem~\ref{theorem-Ding and Chen}. We denote the class of trivial modules by $\class{W}$.
 \end{definition}

Note that if a Ding-Chen ring $R$ happens to be left and right Noetherian, then it is automatically Iwanaga-Gorenstein. In this case $M \in \class{W}$ iff id($M$) is finite iff pd($M$) is finite iff fd($M$) is finite and in this case these dimensions all must be less than or equal to $n$. (See~\cite{iwanaga} or~\cite{enochs-jenda-book} for a proof.) Furthermore, from the Remarks in Section~\ref{sec-ding modules}, the Ding injectives coincide with the Gorenstein injectives and the Ding projectives coincide with the Gorenstein projectives.

\begin{corollary}\label{cor-properties of trivial modules}
Let $R$ be a Ding-Chen ring. Then the class $\class{W}$ of trivial $R$-modules is a thick subcategory. This means $\class{W}$ is closed under retracts and if two out of three terms in a short exact sequence are in $\class{W}$ then so is the third. Furthermore $\class{W}$ is closed under all filtered colimits, in particular transfinite extensions.
\end{corollary}

\begin{proof}
Use Ding and Chen's Theorem~\ref{theorem-Ding and Chen} above along with properties of $\Ext$ to argue that $\class{W}$ is a thick subctegory. The filtered colimits argument is exactly as Hovey proves on page 581 of~\cite{hovey}.
\end{proof}

It is shown in Theorem~3.4 of~\cite{ding and mao 07} that $(\class{W}, \rightperp{\class{W}})$ is a complete cotorsion theory when $R$ is a Ding-Chen ring. It is clear that $\class{W}$ contains $\class{I}$ and is coresolving by Corollary~\ref{cor-properties of trivial modules}. So by Lemma~\ref{lemma-intersecting cotorsion classes} we have $\class{W} \cap \rightperp{\class{W}} = \class{I}$. We see now that $\rightperp{\class{W}} = \class{D}\class{I}$.

We note that in the particular case of when $R$ is an $n$-FC ring, the authors of~\cite{ding and mao 07} call the modules in $\rightperp{\class{W}}$ ``$n$-cotorsion'' modules.

\begin{corollary}\label{cor-characterization of FC injectives}
Let $R$ be a Ding-Chen ring and let $\class{W}$ be the class of trivial modules. Then $N$ is Ding injective if and only if $N \in \rightperp{\class{W}}$.
\end{corollary}

\begin{proof}
Suppose $N$ is Ding injective. Let $W \in \class{W}$. We wish to show $\Ext^1_R(W,N) = 0$. Write a finite FP-injective coresolution $0 \ar W \ar E^0 \ar E^1 \ar \cdots \ar E^n \ar 0$. Then using Lemma~\ref{lemma-first Ext lemma} and dimension shifting we get $\Ext^1_R(W,N) \cong \Ext^{n+1}_R(E^n,N) = 0$.

On the other hand suppose that $N \in \rightperp{\class{W}}$. We wish to show that $N$ is Ding injective. First
take an injective coresolution of $N$ as below.
$$0 \ar N \ar I^0 \ar I^1 \ar I^2 \cdots $$ Note that since $N \in \rightperp{\class{W}}$, and $(\class{W}, \rightperp{\class{W}})$ is a hereditary cotorsion pair, the kernel at any spot in the sequence is also
in $\rightperp{\class{W}}$.
Next we use the fact that $(\class{W},\rightperp{\class{W}})$ is a
complete cotorsion pair to find a short exact sequence $0 \ar K \ar I_0 \ar N \ar 0$
where $I_0 \in \class{W}$ and $K \in \rightperp{\class{W}}$.
But $I_0$ must also be in $\rightperp{\class{W}}$
since it is an extension of two such modules.
As noted before the statement of Corollary~\ref{cor-characterization of FC injectives}, it follows that $I_0$ is an injective module.
Continuing with the same procedure on $K$ we can build an injective resolution of $N$
as below: $$ \cdots \ar I_1 \ar I_0 \ar N \ar 0.$$ Again the kernel at
each spot is in $\rightperp{\class{W}}$. Pasting this ``left'' resolution together
with the ``right'' coresolution above we get an exact
sequence $$\cdots \ar I_1 \ar I_0 \ar I^0 \ar I^1 \ar \cdots $$ of
injective modules with $N = \ker{(I^0 \ar I^1)}$. This sequence satisfies the definition of $N$ being a Ding injective $R$-module since now $\Hom_R(E, -)$ will leave the sequence exact for any FP-injective module $E$.
\end{proof}

We have a dual result as well.
If $R$ is a Ding-Chen ring, then by Theorem~3.8 of~\cite{ding and mao 05}, $(\leftperp{\class{W}}, \class{W})$ is a complete cotorsion pair. It is clear that $\class{W}$ contains $\class{P}$ and is a resolving class. So by Lemma~\ref{lemma-intersecting cotorsion classes} we have $\leftperp{\class{W}} \cap \class{W} = \class{P}$. We see now that $\leftperp{\class{W}} = \class{D}\class{P}$.

We note that in the particular case of when $R$ is $n$-FC, the authors of~\cite{ding and mao 05} call the modules in $\leftperp{\class{W}}$ ``$n$-FP-projective'' modules.

\begin{corollary}\label{cor-characterization of FC projectives}
Let $R$ be a Ding-Chen ring and let $\class{W}$ be the class of trivial modules. Then $M$ is Ding projective if and only if $M \in \leftperp{\class{W}}$.
\end{corollary}

\begin{proof}
Suppose $M$ is Ding projective. Let $W \in \class{W}$. We wish to show $\Ext^1_R(M,W) = 0$. Write a finite flat resolution $0 \ar F_n \ar \cdots \ar F_1 \ar F_0 \ar W \ar 0$. Then using Lemma~\ref{lemma-second Ext lemma} and dimension shifting we get $\Ext^1_R(M,W) \cong \Ext^{n+1}_R(M,F_n) = 0$.

On the other hand suppose that $M \in \leftperp{\class{W}}$. We wish to show that $M$ is Ding projective. First
take a projective resolution of $M$ as below.
$$\cdots \ar P_2 \ar P_1 \ar P_0 \ar M \ar 0 $$ Note that since $M \in \leftperp{\class{W}}$, and $(\leftperp{\class{W}}, \class{W})$ is a hereditary cotorsion pair, the kernel at any spot in the sequence is also
in $\leftperp{\class{W}}$.
Next we use the fact that $(\leftperp{\class{W}}, \class{W})$ is a
complete cotorsion pair to find a short exact sequence $0 \ar M \ar P^0 \ar C \ar 0$
where $P^0 \in \class{W}$ and $C \in \leftperp{\class{W}}$.
But $P^0$ must also be in $\leftperp{\class{W}}$
since it is an extension of two such modules.
It follows that $P^0$ is a projective module.
Continuing with the same procedure on $C$ we can build a projective coresolution of $M$
as below: $$0 \ar M \ar P^0 \ar P^1 \ar P^2 \ar \cdots.$$ Again the kernel at
each spot is in $\leftperp{\class{W}}$. Pasting this ``right'' coresolution together
with the ``left'' resolution above we get an exact
sequence $$\cdots \ar P_1 \ar P_0 \ar P^0 \ar P^1 \ar \cdots $$ of
projective modules which satisfies the definition of $M$ being a Ding projective $R$-module.
\end{proof}

We now state our results on the existence of three Quillen equivalent homotopy theories on modules over a Ding-Chen ring. 

\begin{theorem}\label{theorem-Ding ring model structures}
Let $R$ be a Ding-Chen ring. Then there are two cofibrantly generated abelian model structures on $\text{Mod-}R$ each having $\class{W}$ as the class of trivial objects. In the first model structure, each module is cofibrant while the fibrant objects (resp. trivially fibrant objects) are the Ding injective modules (resp. injective modules). In the second model structure, each module is fibrant while the cofibrant objects (resp. trivially cofibrant objects) are the Ding projective modules (resp. projective modules). We call the first model structure the \emph{injective model structure} and the second the \emph{projective model structure}. When $R$ is Noetherian these model structures coincide with those in~\cite{hovey}.
\end{theorem}

\begin{proof}
$\class{D}\class{P} \cap \class{W} = \class{P}$ and $\class{D}\class{I} \cap \class{W} = \class{I}$, and $(\class{D}\class{P},\class{W})$ and $(\class{W},\class{D}\class{I})$ are complete hereditary cotorsion pairs (each cogenerated by a set). So the result follows from Theorem~2.2 and Corollary~6.8 of~\cite{hovey}.
\end{proof}

\begin{remark}Note that if $R$ is coherent and has weak dimension $n < \infty$ then $R$ is automatically a Ding-Chen ring. However, in this case $\class{W}$ is the class of all modules. So the model structures in Theorem~\ref{theorem-Ding ring model structures} will be trivial. This has an analog in the Noetherian case: When $R$ is Noetherian with global dimension $n < \infty$ then $R$ is automatically a Gorenstein ring. However, the model structures are trivial. So we are interested in examples of Ding-Chen rings with infinite weak dimension.
\end{remark}

\subsection{The flat model structure}

For any ring $R$, the class $\class{F}$ of flat
$R$-modules form the left side of a complete hereditary cotorsion pair
$(\class{F} ,\class{C})$.  The modules in $\class{C}$ are called cotorsion modules. Similarly, when $R$ is a Ding-Chen ring, the Gorenstein flat
modules form the left side of a complete cotorsion theory, denoted $(\class{G}\class{F},\class{G}\class{C})$. The modules in
$\class{G}\class{C}$ are called Gorenstein cotorsion modules.

\begin{lemma}\label{lemma-ding flats form a cotorsion pair}
When $R$ is a Ding-Chen ring, the Gorenstien flat left $R$-modules form the left side of a complete hereditary cotorsion pair $(\class{G}\class{F},\class{G}\class{C})$.
\end{lemma}

\begin{proof}
It is easy to show that $\class{G}\class{F}$ is resolving by using Theorem~5 of~\cite{ding and chen 96} and arguing with the functors $\Tor^R_i(E,-)$ where $E$ is an arbitrary FP-injective right $R$-module and $i \geq 1$. A similar $\Tor$ argument will show that if $0 \ar P \ar F \ar F/P \ar 0$ is a pure exact sequence with $F \in \class{G}\class{F}$, then $F/P$, and consequently $P$, are also in $\class{G}\class{F}$. Now an argument just like Proposition~7.4.3 of~\cite{enochs-jenda-book} will show that $(\class{G}\class{F},\class{G}\class{C})$ is cogenerated by a set and hence is complete.
\end{proof}

\begin{theorem}\label{theorem-Gorenstein flat model structure}
If $R$ is a Ding-Chen ring, there is a model structure on $R$-Mod in
which the cofibrant objects are the Gorenstein flat modules, the
fibrant objects are the cotorsion modules, and the trivial
objects are the modules of finite FP-injective dimension.
\end{theorem}

\begin{proof}
We know that $(\class{F},\class{C})$ and $(\class{G}\class{F},\class{G}\class{C})$ are each complete hereditary cotorsion pairs. Again, let $\class{W}$ denote the class of modules with finite FP-injective dimension. Using the results of~\cite{hovey}, the only thing to check is
that $\class{G}\class{F}\cap \class{W} = \class{F}$ and
$\class{C} \cap \class{W} = \class{G}\class{C}$.  It is clear that
$\class{F} \subseteq \class{G}\class{F} \cap \class{W}$.  Recall
that $F$ is a flat left $R$-module if and only if $F^{+}=\Hom_{\Z} (F,\Q)$ is injective
as a right $R$-module.  Thus, if $F \in \class{W}$, $F$ has finite
flat dimension, so $F^{+}$ has finite injective dimension.  Also, by Theorem~5 of~\cite{ding and chen 96}, $F$ is Gorenstein flat if and only if $F^+$ is Gorenstein injective when $R$ is a Ding-Chen ring.
Since we know $\class{G}\class{I} \cap
\class{W}$ is the class of injectives, we
conclude that $F \in \class{G}\class{F} \cap \class{W}$
implies $F^+$ is injective. But this means that $F$ must be flat. So
$\class{F} = \class{G}\class{F} \cap \class{W}$.

Now by Proposition~\ref{prop-ding projective implies ding flat} it follows that $\class{D}\class{P} \subseteq \class{G}\class{F}$, and of course $\class{F} \subseteq \class{G}\class{F}$. So taking the right half of the associated
cotorsion theories gives us $\class{G}\class{C} \subseteq
\class{C} \cap \class{W}$.  Now suppose $X \in \class{C} \cap
\class{W}$. Since $(\class{G}\class{F},\class{G}\class{C})$ is
complete we can find a short exact sequence $0 \ar X \ar C \ar F \ar
0$ where $C \in \class{G}\class{C}$ and $F \in
\class{G}\class{F}$. Since both $X$ and $C$ are in $\class{W}$, so
is $F$.  By the last paragraph, we see that $F \in \class{F}$, which
makes $0 \ar X \ar C \ar F \ar 0$ split and so $X$ is a summand of
$C$. Therefore $X \in \class{G}\class{C}$. So $\class{C} \cap \class{W} = \class{G}\class{C}$
\end{proof}


\section{Applications to group rings}\label{sec-applications to group rings}

As we commented in Section~\ref{sec-injective and projective model strucs} we are interested in examples of Ding-Chen rings with infinite weak dimension. A good example comes from considering group rings. These model structures are relevant for defining Tate cohomology.

\begin{proposition}
Let $R$ be a commutative Ding-Chen ring. Then the group ring $R[G]$ is a Ding-Chen ring for any finite group $G$.
\end{proposition}

\begin{proof}
First lets recall some things about the group ring $R[G]$. The ring inclusion $R \hookrightarrow R[G]$ induces, by restricting scalars, a functor $U \mathcolon \text{Mod-}R[G] \ar R\text{-Mod}$. The functor $- \tensor_R R[G] \mathcolon R\text{-Mod} \ar \text{Mod-}R[G]$ is left adjoint to $U$. On the other hand, $\Hom_R(R[G],-)$ is right adjoint to $U$. (Here the right $R[G]$-module structure on $\Hom_R(R[G],M)$ is define by $(fr)(s) = f(rs)$.) Since $G$ is finite we have $\Hom_R(R[G],M) \cong M \tensor_R R[G]$ as right $R[G]$-modules. Similarly, we have a forgetful functor $V \mathcolon R[G]\text{-Mod} \ar R\text{-Mod}$ with left adjoint functor $R[G] \tensor_R - \cong \Hom_R(R[G],-)$. (This time the left $R[G]$-module structure on $\Hom_R(R[G],M)$ is define by $(rf)(s) = f(sr)$.)

First we show that $R[G]$ is right coherent. For any indexing set $I$, the product $\prod_{i \in I} R$ is a flat $R$-module by Chase's Theorem (Theorem~4.47 of~\cite{lam}). Since $- \tensor_R R[G]$ is left adjoint to an exact functor it preserves flat modules. Therefore  $(\prod_{i \in I} R) \tensor_R R[G]$ is a flat right $R[G]$-module. But since $G$ is finite, $- \tensor_R R[G]$ is also a right adjoint, so $(\prod_{i \in I} R) \tensor_R R[G] \cong \prod_{i \in I} (R \tensor_R R[G]) \cong \prod_{i \in I} R[G]$ is a flat right $R[G]$-module. It follows, again from Chase's Theorem, that $R[G]$ is left coherent. Arguing on the other side with $R[G] \tensor_R -$ shows $R[G]$ is right coherent as well.

Note that since $G$ is finite, $U$ preserves finitely generated modules. Since $U$ is exact it also preserves finitely presented modules. Since $\Hom_R(R[G],-) \mathcolon R\text{-Mod} \ar \text{Mod-}R[G]$ is also exact, one can prove $$\Ext^n_R(U(M),N) \cong \Ext^n_{R[G]}(M,\Hom_R(R[G],N)).$$ From these two observations, it follows immediately that $\Hom_R(R[G],-)$ preserves FP-injective modules. Now since $R[G]$ is coherent and $\Hom_R(R[G],-)$ preserves FP-injective modules, it follows from Lemma~3.1 of~\cite{stenstrom-FP-injective} that $\Hom_R(R[G],-)$ preserves modules of finite FP-injective dimension. Thus $R[G] \cong \Hom_R(R[G],R)$ has finite right self FP-injective dimension whenever $R$ is a commutative Ding-Chen ring. Arguing on the other side with $V$ and $\Hom_R(R[G],-) \mathcolon R\text{-Mod} \ar R[G]\text{-Mod}$ we see that $R[G]$ had finite left self FP-injective dimension whenever $R$ is a commutative Ding-Chen ring.
\end{proof}

Damiano points out in~\cite{damiano} the following result of Colby. A group ring $R[G]$ is an FC ring if and only if $R$ is an FC ring and $G$ is a locally finite group. This raises the following question.

\begin{question}
Let $R$ be a commutative ring. Is the group ring $R[G]$ a Ding-Chen ring if and only if $R$ is a Ding-Chen ring and $G$ is a locally finite group?
\end{question}


\end{document}